\documentclass[a4paper]{article}
\usepackage{amsmath}
\usepackage{indentfirst}
\usepackage{geometry}
\usepackage{amsfonts}
\usepackage{amssymb}
\usepackage{amsthm}
\newtheorem{Theorem}{Theorem}[section]
\newtheorem{Remark}[Theorem]{Remark}
\newtheorem{Lemma}[Theorem]{Lemma}

\newtheorem{Proposition}[Theorem]{Proposition}
\numberwithin{equation}{section}
\newcommand{\norm}[1]{\left\Vert#1\right\Vert}
\newcommand{\bb}[1]{\mathbb{#1}}
\newcommand{\pp}[1]{\textbf{\emph{#1}}}

\title{Almost Sure Existence of Global Weak Solutions to the 3D\\ Incompressible Navier-Stokes Equation}
\author{Jingrui Wang\footnotemark[1] \and Keyan Wang\footnotemark[2]}

\begin{document}
\maketitle
\begin{abstract}
In this paper we prove the almost sure existence of global weak solution to the 3D incompressible Navier-Stokes Equation for a set of large data in $\dot{H}^{-\alpha}(\mathbb{R}^{3})$ or $\dot{H}^{-\alpha}(\mathbb{T}^{3})$ with $0<\alpha\leq 1/2$. This is achieved by randomizing the initial data and showing that the energy of the solution modulus the linear part keeps finite for all $t\geq0$. Moreover, the energy of the solutions is also finite for all $t>0$. This improves the recent result of Nahmod, Pavlovi\'{c} and Staffilani on (SIMA, \cite{ASE}) in which $\alpha$ is restricted to $0<\alpha<\frac{1}{4}$.
\end{abstract}

\footnotetext[1]{School of Mathematical Sciences and Shanghai Center for Mathematical Sciences, Fudan University, Shanghai,  200433,  P. R. China. \textit{Email: 13110180044@fudan.edu.cn}}
\footnotetext[2]{Department of Mathematics, Shanghai Finance University, Shanghai, 201209, P. R. China. \textit{Email: wang.keyan@yahoo.com}}
\section{Introduction}
Let's recall the incompressible Navier-Stokes equations in the whole space $\bb{R}^{3}$ or the 3D torus $\bb{T}^{3}$,
\begin{equation}
\label{eq:navier}
\begin{cases}
\partial_{t}u + u\cdot\nabla u + \nabla p=\Delta u,\\
\nabla\cdot u=0,\\
u(x,0)=f(x).
\end{cases}
\end{equation}\\
Here $u$ is the velocity vector, $p$ is the scalar pressure and $f$ is the initial data which is divergence-free.

For divergence-free initial data $f \in L^2$, Leray proved the existence of global weak solutions to $(\ref{eq:navier})$ in his seminal work $\cite{LER}$. The result of Leray was then extended to the bounded domain case by Hopf $\cite{HOP}$. See also Lemarie-Rieusset \cite{DNP} for an extension to the uniform $L^2$ initial data. Very recently, for $0<\alpha<\frac{1}{4}$, Nahmod, Pavlovic and Staffilani $\cite{ASE}$ showed the almost sure existence of global weak solutions for supercritical periodic initial data in $\dot{H}^{-\alpha}(\bb{T}^{3})$ after suitable data randomization technique. The aim of this paper is to extend the interesting work in \cite{ASE} to the case when $0 < \alpha \leq \frac{1}{2}$ in the whole 3D space. \\
\indent Before presenting our main result, let us give a brief review on known results which are very related to ours. In $\cite{RDC}$, Burq and Tzvetkov proved the well-posedness of nonlinear wave equations below the critical threshold by randomizing the initial data in an appropriate way. The key point of their method consists in using the fact that though the initial data have low regularity, their Lebesgue integrability is almost surely better than expected. Such kind of $L^{p}$ regularization phenomena are in fact well-known since the work of Paley-Zygmund \cite{FOR}.

The method in \cite{RDC} was extensively applied to other PDE problems when the regularity of the initial data is below the critical threshold. For instance, Burq and Tzvetkov in \cite{RDS,PWP} studied the global existence of cubic semi-linear wave equations. Zhang and Fang $\cite{RCG}$ applied this approach for Navier-Stokes Equation and obtained the local smooth solutions for the randomized initial data in $L^{2}(\bb{T}^{3})$. Later on, they further relaxed the constraint on the initial data to $\dot{H}^{-1+}(\bb{R}^{3})$ in $\cite{WHO}$. We remark that the solution obtained in $\cite{WHO}$ is not in the energy space when $t>0$. Deng and Cui $\cite{PNS}$ proved the global existence of classical solutions for randomized small initial data in $L^{2}(\bb{T}^{3})$.

To state our main result, let us first introduce some basic notations and lay down the randomization set-up. Let's first denote
$$D=\sqrt{-\Delta}.$$
Starting with the work of Bourgain $\cite{IMS}$ and Burq and Tzvetkov $\cite{RDC}$, there are many results for probabilistic constructions on a compact manifold $M$, where there is a countable basis $\{e_{n}(x)\}$ in $L^{2}(M)$ consisting the eigenfunctions of the Laplace-Beltrami operator. We can use this sequence $\{e_{n}(x)\}|_{n=1}^{\infty}$ to introduce such a randomization: Given $D^{-\alpha}f(x)=\sum_{n}a_{n}e_{n}(x)$, where $f$ is in the space $\dot{H}^{-\alpha}$, one can define the randomization by
\begin{equation}
\label{eq:torus}
D^{-\alpha}f^{\omega}(x)=\sum_{n}a_{n}e_{n}(x)h_{n}(\omega),
\end{equation}
where $\{h_{n}(\omega)\}$ are a series of independent mean zero real random variables with identical Gaussian distribution on a probability space $(\Omega,\mathcal{A},\mathcal{P})$. \\

In the whole space case, we follow the way in Zhang and Fang $\cite{WHO}$, see also L\"{u}hrmann and Mendelson $\cite{DEN}$ and \'{A}. B\'{e}nyi, T. Oh and O. Pocovinicu $\cite{ONA,BEN}$, dividing the frequency space in accordance with the \textit{Wiener decomposition} instead of eigenfunctions. Let $n\in \bb{Z}^{3}$ and $Q_{n}$ be the unit cube $Q_{n}=n+(-\frac{1}{2},\frac{1}{2}]^{3}$. Thus we have
$$\bb{R}^{3}=\bigcup_{n}Q_{n}.$$
Note that $Q_{n}\cap Q_{m}=\emptyset$ if $n\neq m$, we have $\sum_{n}\chi_{Q_{n}}(\xi)=1$. Hence we get the decomposition
$$
f(x)=\sum_{n\in\bb{Z}^{3}}\mathcal{F}^{-1}(\chi_{Q_{n}}\widehat{f}).
$$
Here $\widehat{f}$ denotes the usual Fourier transformation of $f$ which is also denoted by $\mathcal{F}(f)$, and $\mathcal{F}^{-1}(g)$ denotes the inverse Fourier transformation of $g$.
This partition is commonly referred to as \textit{Wiener decomposition}. Moreover, as $\cite{ONA,BEN,DEN}$, we have a smooth version for the decomposition. Define a nonnegative and even function $\phi\in C_{c}^{\infty}(\bb{R}^{3})$, and
$$
\phi(\xi)=\begin{cases}
1,\quad \xi\in(-\frac{1}{2},\frac{1}{2})^{3}\\
0,\quad \xi\in([-1,1]^{3})^{c}.
\end{cases}
$$
Let
$$
\varphi(\xi)=\frac{\phi(\xi)}{\sum_{n\in\bb{Z}^{3}}\phi(\xi-n)}.
$$
We can see from the definition above that $\varphi$ is an even valued function and supp$\varphi\subset [-1,1]^{3}$. Especially, $\sum_{n}\varphi(\xi-n)=1$.
Define
\begin{equation}
\label{eq:decomposition}
\varphi(D-n)f = \int_{\bb{R}^{3}}\widehat{f}(\xi)\varphi(\xi-n)e^{2\pi ix\xi}d\xi
\end{equation}
Hence $f$ has such a decomposition
$$
f(x)=\sum_{n\in\bb{Z}^{3}}\varphi(D-n)f.
$$
We note that $\varphi$ is an even function, also with the fact that $\overline{\widehat{f}(\xi)}=\widehat{f}(-\xi)$. Then for any real-valued $f$ we have $$\overline{\varphi(D+n)f}=\varphi(D-n)f.$$
Therefore we have $\sum_{n\in\bb{Z}^{3}}\varphi(D-n)f$ is still a real-valued function. This fact ensures that the decomposition $(\ref{eq:decomposition})$ is closed within real-valued functions.\\

We then define
\begin{equation}
\label{eq:fomega}
D^{-\alpha}f^{\omega}(x)=\sum_{n\in\bb{Z}^{3}}h_{n}(\omega)\varphi(D-n)D^{-\alpha}f(x).
\end{equation}
We then let $f^{\omega}$ be the randomized initial data defined in $(\ref{eq:torus})$ or $(\ref{eq:fomega})$, and consider the Cauchy problem of the following Navier-Stokes equations
\begin{equation}
\label{eq:randomsystem}
\begin{cases}
\partial_{t}u=\Delta u-\pp{P}\nabla\cdot(u\otimes u),\\
\nabla\cdot u=0,\\
u(x,0)=f^{\omega}(x).
\end{cases}
\end{equation}
Here $\pp{P}$ is the projection to the divergent free vector space which is defined as
\begin{equation}
\pp{P} u = u-\nabla(-\Delta)^{-1}(\nabla\cdot u).
\end{equation}

Now we are ready to state our main result.

\begin{Theorem}
\label{thm}
Let $0 < T \leq \infty$, $0< \alpha\leq \frac{1}{2}$ and let $f\in \dot{H}^{-\alpha}(\bb{R}^{3})$ or $f\in H^{-\alpha}(\bb{T}^{3})$ be divergence-free. We further assume that $f$ is of mean zero in the periodic case. Then there exists a set $\Sigma\subset\Omega$ of probability 1 such that for any $\omega\in\Sigma$, the random data Navier-Stokes equation $(\ref{eq:randomsystem})$ has a global weak solution $u$
with
$$
u = g+v,
$$
where $g=e^{t\Delta}f^{\omega}$ and $v\in L^{\infty}([0,T],L^{2}_{x})\cap L^{2}([0,T],\dot{H}^{1}_{x})$, and $u\in L^{\infty}([\delta_{0},T],L^{2}_{x})\cap L^{2}([\delta_{0},T],\dot{H}^{1}_{x})$ for every $\delta_{0}>0$. Moreover, define
\begin{equation}
\label{energy}
E(v,t)=\norm{v(\cdot,t)}_{L^{2}}^{2}+\norm{v}^{2}_{L^{2}([0,t],\dot{H}^{1})},
\end{equation}
and we have
$$
\sup_{t\geq 0}E(v,t)\leq C(\alpha,\norm{f}_{\dot{H}^{-\alpha}}),\quad \sup_{t\geq \delta_{0}}E(u,t)\leq C(\alpha,\delta_{0},\norm{f}_{\dot{H}^{-\alpha}}).
$$
\end{Theorem}

\indent To prove Theorem $\ref{thm}$, as in \cite{ASE}, we also write $u = g + v$ and then study the equation governing $v$. Clearly, the initial data of $v$ vanishes and thus is smooth enough, but there are inputs to the equation of $v$ coming from the linear part $g$. We try to utilize the improved space-time Lebesgue regularity of the linear heat solution $g = e^{t\Delta}f^\omega$ due to the randomization. Compared with \cite{ASE}, the key improvement in our paper is to derive that the linear evolution $g=e^{t\Delta}f^{\omega}$ has an almost sure $L^{p}_{t}L^{q}_{x}$ bound, which can be  small enough on some short time interval for any given $2\leq p, q< \infty$ and any given $0<\alpha<1$ with $\alpha q<2$. See Lemma $\ref{lem:est}$ for details. For $0< \alpha \leq \frac{1}{2}$, using this smallness of $L^3_{t}L^9_{x}$ norm and the boundness of $L^{4}_{t}L^4_{x}$ norm of $g$, we are able to derive the energy bound for $v$ on a short time interval. On the other hand, for a.e. $t_0 > 0$, $u(t_0, \cdot)$ will be a $L^2$ function. Then we may use the classical result by Leray \cite{LER} to extend this weak solution to be a global one in time.

\indent
The remaining part of this paper is organized as follows. In section 2 we will first recall several important properties of randomization that we will use. Then we make our probabilistic estimate and lay down some probabilistic set-up of our conclusion. Then we derive the energy bound for $v$ in section 3, and the proof of Theorem $\ref{thm}$ is also given at the end of section 3. In section 4 we will construct the global weak solution for the Navier-Stokes equation for randomized initial data $f^{\omega}$.\\

\section{Average effects}
Let's begin this section with a large derivation bound property which is proved on Lemma 3.1 in \cite{RDC}.
\begin{Lemma}
\label{lem:prob}
Let $(h_{i}(\omega))$ be a sequence of real, mean zero and independent random variables on a probability space$(\Omega,\mathcal{A},\mathcal{P})$. Let $\mu_{i}$ be the associatied distributions. Assume that $\mu_{i}$ satisfy the following property
\begin{equation}
\label{eq:burq}
\exists c>0:\forall\gamma\in\bb{R},\forall i\in\mathbb{N}^{+},\big|\int_{-\infty}^{\infty}e^{\gamma x}d\mu_{i}\big|\leq e^{c\gamma^{2}}.
\end{equation}
Then there exists $\beta>0$ such that for every $\lambda>0$, every sequence $\{c_{i}\}\in l^{2}$ of real numbers,
$$
\mathcal{P}\left(\omega:\big|\sum_{i=1}^{\infty}c_{i}h_{i}(\omega)\big|>\lambda\right)\leq 2e^{-\frac{\beta\lambda^{2}}{\sum_{i=1}c_{i}^{2}}}.
$$
Moreover, there exists $C>0$ such that for every $r\geq2$ and every $(c_{i})\in l^{2}$,
\begin{equation}
\label{eq:gain}
\norm{\sum_{i=1}^{\infty}c_{i}h_{i}(\omega)}_{L^{r}(\Omega)}\leq C\sqrt{r}\left(\sum_{i=1}^{\infty}c_{i}^{2}\right)^{\frac{1}{2}}.
\end{equation}
\end{Lemma}

\begin{Remark}
Burq and Tzvetkov have shown in $\cite{RDC}$ that the Gaussian distribution satisfies the assumption $(\ref{eq:burq})$ so that this lemma can be used in our randomization. Moreover, by the result of Lemma $\ref{lem:prob}$ we have
$$
\mathcal{P}\left(\omega:\norm{f^{\omega}}_{\dot{H}^{-\alpha}}>\lambda\right)\leq 2\exp(-\frac{\beta\lambda^{2}}{\norm{f}_{\dot{H}^{-\alpha}}}).
$$
Hence, we can see that $f^{\omega}$ is in $\dot{H}^{-\alpha}$ almost surely provided $f$ is in $\dot{H}^{-\alpha}$.
\end{Remark}

Now we turn to the estimate of the linear part $g$. We first consider the deterministic estimate of $g$ which are standard and well-known. For instance, the periodic case is proved in Lemma 3.1 of $\cite{ASE}$. The whole space case can be treated in a similar way.
\begin{Lemma}
Let $0<\alpha\leq\frac{1}{2}$, let k be a nonnegative integer and $g=e^{t\Delta}f^{\omega}$. If $f^{\omega}\in \dot{H}^{-\alpha}$, then we have
\begin{equation}
\label{determ}
\begin{cases}
\norm{D^{k}g(\cdot,t)}_{L^{2}}\leq Ct^{-\frac{\alpha+k}{2}}\norm{f}_{\dot{H}^{-\alpha}},\\
\norm{D^{k}g(\cdot,t)}_{L^{\infty}}\leq C(\max\{t^{-\frac{1}{2}}, t^{-\frac{\alpha+k+d/2}{2}}\})\norm{f}_{\dot{H}^{-\alpha}}.
\end{cases}
\end{equation}
\end{Lemma}
\indent This deterministic estimate directly provides the smoothness of the linear evolution of $g$ as $t>0$. Next we will follow the approach applied in \cite{RDC}, using the average effect to make an $L^{r}_{\omega}L^{p}_{T}L^{q}_{x}$-estimate of $g$, and we will see the linear part determined by randomized initial data does provide an improved integral estimate although the randomization introduces no Sobolev regularity on initial data. \\
\indent Before we continue further, let's make a brief recall to the smooth projections in the Littlewood-Paley theory. Let $\rho$ be a smooth real-valued bump function supported in $\{\xi\in \mathbb{R}^{3}:|\xi|\leq 2\}$ and $\rho=1$ on $\{\xi:|\xi|\leq 1\}$. If $M\geq 1$ is a dyadic number, we define $P_{\leq M}$ onto the truncated space $\{|\xi|\leq M\}$ as
$$
\widehat{P_{\leq M}v}(\xi)=\widehat{v}(\xi)\rho(\xi/M).
$$
Similarly, we can define the smooth projector $P_{M}$ onto the truncated space $\{|\xi|\sim M\}$ as
\begin{equation}
\label{eq:projector}
\widehat{P_{M}v}(\xi)=\widehat{v}(\xi)(\rho(\xi/M)-\rho(2\xi/M)).
\end{equation}
In our decomposition
$$
f= \sum_{n\in\bb{Z}^{3}}\varphi(D-n)f,
$$
the function $\varphi(D-n)f$ has a bounded frequency support in $n+[-1,1]^{3}$. For these functions, we have the following classical Bernstein inequalities.
\begin{Lemma}[\textbf{Bernstein}]
Let $\mathcal{B}$ be a ball. A constant $C$ exists such that for any nonnegative integer $k$, any couple $(p,q)$ in $[1,\infty]^{2}$ with $1\leq p\leq q$, and any function $\psi\in L^{p}$ with supp$\widehat{\psi}\in N\mathcal{B}$, we have
\begin{equation}
\norm{\psi}_{L^{q}}\leq CN^{3(\frac{1}{p}-\frac{1}{q})}\norm{\psi}_{L^{p}}.
\end{equation}
\end{Lemma}
Let $\psi(x)=e^{2\pi inx}\varphi(D-n)f(x)$. We can see that supp $\widehat{\psi}\subset [-1,1]^{3}$. Apply the Bernstein inequality we have
$$
\norm{e^{2\pi inx}\varphi(D-n)f(x)}_{L^{q}_{x}}\leq C\norm{e^{2\pi inx}\varphi(D-n)f(x)}_{L^{p}_{x}}.
$$
Let $p=2$ and $q\geq 2$, we have
\begin{equation}
\label{eq:Bernstein}
\norm{\varphi(D-n)f(x)}_{L^{q}_{x}}\leq C\norm{\varphi(D-n)f(x)}_{L^{2}_{x}}.
\end{equation}

We have the probabilistic estimate for the linear part $g$ for both the periodic space case and the whole space case.
\begin{Proposition}
\label{lem:est}
Let $T>0$, $0< \alpha < 1$, let $p,q$ satisfy $2\leq p\leq r<\infty$ and $2\leq q\leq r<\infty$. If $\alpha p\leq 2$, then there exists $C_{p,q,\alpha}>0$ such that
\begin{equation}
\label{eq:prob}
\norm{e^{t\Delta}f^{\omega}}_{L^{r}(\Omega;L^{p}([0,T],L^{q}_{x}))}\leq
C_{p,q,\alpha}T^{\frac{1}{p}-\frac{\alpha}{2}}\norm{f}_{\dot{H}^{-\alpha}},
\end{equation}
for every $f\in \dot{H}^{-\alpha}(\bb{R}^{3})$ and every mean zero $f\in H^{-\alpha}(\bb{T}^{3})$.
\end{Proposition}
\begin{proof}
We first consider the periodic space case.
First using $(\ref{eq:torus})$, we can write
$$
e^{t\Delta}f^{\omega}
=D^{\alpha}e^{t\Delta}\sum_{n}a_{n}h_{n}(\omega)e_{n}(x).
$$
Since $2\leq p,q\leq r$, by Minkovski inequality and Lemma $\ref{lem:prob}$, we can get
\begin{align*}
&\big\|D^{\alpha}e^{t\Delta}\sum_{n}a_{n}h_{n}(\omega)e_{n}(x)\big\|_{L^{r}(\Omega;L^{p}([0,T];L^{q}_{x}(\bb{T}^{3})))}\\
\leq & \big\|\sum_{n}a_{n}D^{\alpha}e^{t\Delta}e_{n}(x)h_{n}(\omega)\big\|_{L^{p}([0,T];L^{q}_{x}(\bb{T}^{3};L^{r}(\Omega)))}\\
\leq &\sqrt{r}\norm{a_{n}D^{\alpha}e^{t\Delta}e_{n}(x)}_{L^{p}([0,T];L^{q}_{x}(\bb{T}^{3});l^{2}_{n})}\\
\leq &\sqrt{r}\norm{\norm{a_{n}D^{\alpha}e^{t\Delta}e_{n}(x)}_{L^{p}([0,T];L^{q}_{x}(\bb{T}^{3}))}}_{l^{2}_{n}}.
\end{align*}
We recall the Young inequality which shows that when $q\geq 2$,
$$
\norm{f}_{L^{q}}\leq C_q\norm{\widehat{f}}_{l^{q^{'}}},\quad \frac{1}{q}+\frac{1}{q'}=1.
$$
Also with the help of Minkovski inequality
\begin{align*}
&\norm{\norm{a_{n}D^{\alpha}e^{t\Delta}e_{n}(x)}_{L^{p}([0,T];L^{q}_{x}(\bb{T}^{3}))}}_{l^{2}_{n}}\\
\leq&C_{q}\norm{a_{n}\norm{|m|^{\alpha}e^{-t|m|^{2}}\widehat{e_{n}}(m)}_{L^{p}([0,T];l^{q{'}}_{m})}}_{l^{2}_{n}}\\
=&C_{q}\norm{a_{n}|n|^{\alpha}\norm{e^{-|n|^{2}t}}_{L^{p}[0,T]}}_{l^{2}_{n}}\\
\leq&C_{q}\norm{a_{n}\sup_{|n|\in\bb{Z}^{3}/\{0\}}|n|^{\alpha-\frac{2}{p}}(1-e^{-p|n|^{2}T})^{\frac{1}{p}}}_{l^{2}_{n}}\\
\end{align*}
Let's denote
$$J(T)=\sup_{\xi\in\bb{R}^{3}}|\xi|^{\alpha-\frac{2}{p}}(1-e^{-p|\xi|^{2}T})^{\frac{1}{p}},$$
thus
\begin{align*}
&\big\|D^{\alpha}e^{t\Delta}\sum_{n}a_{n}h_{n}(\omega)e_{n}(x)\big\|_{L^{r}(\Omega;L^{p}([0,T];L^{q}_{x}(\bb{T}^{3})))}\\
\leq &C_{q,r}\sup_{|n|\in\bb{Z}^{3}/\{0\}}|n|^{\alpha-\frac{2}{p}}(1-e^{-|n|^{2}Tp})^{\frac{1}{p}}\norm{a_{n}}_{l^{2}_{n}}\\
\leq &C_{q,r}J(T)\norm{f}_{\dot{H}^{-\alpha}}.
\end{align*}\\

For the whole space case, using $(\ref{eq:fomega})$, we can write
\begin{align*}
e^{t\Delta}f^{\omega}
&=e^{t\Delta}\sum_{n\in\bb{Z}^{3}}h_{n}(\omega)\varphi(D-n)f(x)
\end{align*}
Since $2\leq p,q\leq r$, by Minkovski inequality and Lemma $\ref{lem:prob}$, we can get
\begin{align*}
&\big\|e^{t\Delta}\sum_{n\in\bb{Z}^{3}}h_{n}(\omega)\varphi(D-n)f(x)\big\|_{L^{r}(\Omega;L^{p}([0,T];L^{q}_{x}(\bb{R}^{3})))}\\
\leq & \big\|\sum_{n\in\bb{Z}^{3}}h_{n}(\omega)\varphi(D-n)e^{t\Delta}f(x)\big\|_{L^{p}([0,T];L^{q}_{x}(\bb{R}^{3};L^{r}(\Omega)))}\\
\leq &\sqrt{r}\norm{\varphi(D-n)e^{t\Delta}f(x)}_{L^{p}([0,T];L^{q}_{x}(\bb{R}^{3});l^{2}_{n})}\\
\leq &\sqrt{r}\norm{\varphi(D-n)e^{t\Delta}f(x)}_{l^{2}_{n};L^{p}([0,T];L^{q}_{x}(\bb{R}^{3}))}.
\end{align*}
Now the Bernstein inequality $(\ref{eq:Bernstein})$ comes in. With Minkovski inequality, we have
\begin{align*}
&\norm{\varphi(D-n)e^{t\Delta}f(x)}_{l^{2}_{n};L^{p}([0,T];L^{q}_{x}(\bb{R}^{3}))}\\
\leq&C_{q}\norm{\varphi(D-n)e^{t\Delta}f(x)}_{l^{2}_{n};L^{p}([0,T];L^{2}_{x}(\bb{R}^{3}))}\\
\leq&C_{q}\norm{\varphi(\xi-n)\widehat{f}(\xi)\norm{e^{-t|\xi|^{2}}}_{L^{p}[0,T]}}_{L^{2}_{\xi}l^{2}_{n}}\\
=&C_{q}\norm{\varphi(\xi-n)\widehat{D^{-\alpha}f}(\xi)|\xi|^{\alpha-\frac{2}{p}}(1-e^{-T|\xi|^{2}p})^{\frac{1}{p}}}_{L^{2}_{\xi}l^{2}_{n}}\\
\leq&C_{q}\sup_{\xi\in R^{3}}(|\xi|^{\alpha-\frac{2}{p}}(1-e^{-T|\xi|^{2}p})^{\frac{1}{p}})\norm{\varphi(\xi-n)
\widehat{D^{-\alpha}f}(\xi)}_{L^{2}_{\xi}l^{2}_{n}}.
\end{align*}
Since $\varphi(\xi)\leq 1$, we have
\begin{align*}
&\big\|e^{t\Delta}\sum_{n\in\bb{Z}^{3}}h_{n}(\omega)\varphi(D-n)f(x)\big\|_{L^{r}(\Omega;L^{p}([0,T];L^{q}_{x}(\bb{R}^{3})))}\\
\leq &C_{r,q}J(T)\norm{\varphi(\xi-n)\widehat{D^{-\alpha}f}(\xi)}_{L^{2}_{\xi}l^{2}_{n}}\\
\leq &C_{r,q}J(T)\norm{\widehat{D^{-\alpha}f}(\xi)}_{L^{2}_{\xi}}\\
=&C_{r,q}J(T)\norm{f}_{\dot{H}^{-\alpha}(\bb{R}^{3})}.
\end{align*}

Now we estimate the term $J(T)$. Define $I(y)$ be the function
$$
I(y)=y^{\alpha-\frac{2}{p}}(1-e^{-y^{2}Tp})^{\frac{1}{p}},\quad (y\geq 0).
$$
and thus
$$
J(T)=\sup_{\xi\in\bb{R}^{3}}I(|\xi|).
$$
It is clear that if $\alpha p=2$, then the Lemma is correct.

Now let us assume that $\alpha p<2$. Make the Tylor expansion around the point $y=0$, we have
$$
I(y)=y^{\alpha-\frac{2}{p}}(y^{2}Tp-y^{4}Tp+o(y^{4}))^{\frac{1}{p}}
=C_{p}y^{\alpha}T^{\frac{1}{p}}+o(y^{\alpha}).
$$
Hence $I(0)=0$. With $\alpha p< 2$, we can see $I(\infty)=0$. Hence We can see that $I(y)$ reaches its maximal only when $y$ reaches the stationary point. With a derivation calculation, we get
$$
[I(y)^{p}]'=0\Leftrightarrow \left(\alpha p-2\right)\left(1-e^{-y^{2}Tp}\right)+2e^{-y^{2}Tp}pTy^{2}=0.
$$
Note that $y^{\alpha}T^{\frac{\alpha}{2}}e^{-y^{2}T}\leq C$, and thus the maximal of $I(y)$ has an upper bound
\begin{align*}
I(y)&\leq y^{\alpha-\frac{2}{p}}e^{-y^{2}T}(2pT)^{\frac{1}{p}}y^{\frac{2}{p}}(\frac{2}{p}-\alpha)^{\frac{1}{p}}\\
&=\left(y^{\alpha}T^{\frac{\alpha}{2}}e^{-y^{2}T}\right)T^{\frac{1}{p}-\frac{\alpha}{2}}\tilde{C}_{\alpha,p}\\
&\leq C_{\alpha,p}T^{\frac{1}{p}-\frac{\alpha}{2}}.
\end{align*}
In our assumption $\alpha p<2$, the term $\frac{1}{p}-\frac{\alpha}{2}$ is positive and hence $T^{\frac{1}{p}-\frac{\alpha}{2}}$ is bounded. Since we can choose $r=\max\{p,q\}$, the inequality $(\ref{eq:prob})$ follows with $J(T)\leq C_{\alpha,p}T^{\frac{1}{p}-\frac{\alpha}{2}}$.
\end{proof}
\begin{Remark}
For convenience, we denote
\begin{equation}
\sigma(q,\alpha)=\frac{1}{p}-\frac{\alpha}{2}.
\end{equation}
we notice that $\sigma(p,\alpha)\geq0$ once the condition $\alpha p\leq 2$ gets satisfied. Hence for a fixed time period $[0,T]$, the estimate of $\norm{e^{t\Delta}f^{\omega}}_{L^{r}(\Omega;L^{p}([0,T],L^{q}_{x}))}$ is bounded. By the result of Lemma $\ref{lem:est}$ we can see that the randomization does provide an improved $L^{p}_{t}L^{q}_{x}$ estimate in the sense of almost sure.
\end{Remark}

With the estimate above, let's define the probabilistic space
\begin{equation}
\label{def:space}
E_{p,q}(\lambda,T)=\{\omega\in\Omega:\norm{e^{t\Delta}f^{\omega}}
_{L^{p}([0,T],L^{q}_{x})}<\lambda\norm{f}_{\dot{H}^{-\alpha}}\}
\end{equation}
Since $h_{n}$ is a series of Gaussian distribution, by Lemma $\ref{lem:prob}$ we have
\begin{equation}
\label{eq:Tchebyshev}
\begin{aligned}
\mathcal{P}\left(E_{p,q}(\lambda,T)^{c}\right)&\leq c \exp\left(-\frac{\lambda^{2}\norm{f}^{2}_{\dot{H}^{-\alpha}}}{C_{\alpha,p}^{2}T^{2\sigma(p,\alpha)}\norm{f}^{2}_{\dot{H}^{-\alpha}}}\right)\\
&=c\exp\left(-\frac{\lambda^{2}}{C^{2}_{\alpha,p}T^{2\sigma(p,\alpha)}}\right)\\
\end{aligned}
\end{equation}
With the equality $(\ref{eq:Tchebyshev})$ let's state the almost-sure argument for the linear evolution $g$.
\begin{Lemma}
\label{lemma1}
Let $0<\alpha <1$ and $T>0$, and let $2\leq p,q <\infty$. If $\alpha p\leq 2$, then there exists a set $\Sigma_{1}\subset\Omega$ such that $\mathcal{P}(\Sigma_{1})=1$, and for every $\omega\in\Sigma_{1}$, we can choose a $M>0$ so that $\omega\in E_{p,q}(M,T)$.
\end{Lemma}
\begin{proof}
Let $\lambda_{j}=2^{j}$, and define $E_{j}=E_{p,q}(\lambda_{j},T)$. We can see that $E_{j}\subset E_{j+1}$, Let's define $\Sigma_{1}=\bigcup_{j}E_{j}$. By $(\ref{eq:Tchebyshev})$, we have
$$
\mathcal{P}(\Sigma_{1})\geq 1-\lim_{j\rightarrow\infty}\mathcal{P}(E_{j}^{c})\geq 1-\lim_{j\rightarrow\infty}\exp{\left(-\frac{\lambda_{j}^{2}}{C^{2}_{\alpha,p}T^{2\sigma(p,\alpha)}}\right)}=1.
$$
We see, for every $\omega\in\Sigma_{1}$, there exists a $j$ such that $\omega\in E_{j}$. Let $M = 2^{j+1}$, and $\omega\in E_{p,q}(M,T)$
\end{proof}
Furthermore, in the case of $\alpha p<2$ we observe that $\sigma(p,\alpha)>0$, which means that the linear part estimate can be governed by $T$. Hence for every $\varepsilon>0$, we can find a $\delta>0$ such that $$\norm{e^{t\Delta}f^{\omega}}_{L^{r}(\Omega;L^{p}([0,\delta],L^{q}_{x}))}<\varepsilon.$$
With this observation, we can get another type of almost-sure argument.
\begin{Lemma}
\label{lemma2}
Let $0<\alpha< 1$, $M>0$, and let $2\leq p,q <\infty$. If $\alpha p<2$, there exists a set $\Sigma_{2}\subset\Omega$ such that $\mathcal{P}(\Sigma_{2})=1$, and for every $\omega\in\Sigma_{2}$, we can choose a $\delta>0$ so that $\omega\in E_{p,q}(M,\delta)$.
\end{Lemma}
\begin{proof}
Let's fix a $M>0$. Similar to Lemma $\ref{lemma1}$, define $\delta_{j}=2^{-j}$, and define $E_{j}=E_{p,q}(M,\delta_{j})$. By the definition $E_{p,q}$ in $(\ref{def:space})$, we can still get $E_{j}\subset E_{j+1}$. Let's define $\Sigma_{2}=\bigcup_{j}E_{j}$. Since $\alpha q<2$, the index of time $\sigma(\alpha,q)>0$. Also with the inequality $(\ref{eq:Tchebyshev})$, we have
$$
\mathcal{P}(\Sigma_{1})\geq 1-\lim_{j\rightarrow\infty}\mathcal{P}(E_{j}^{c})\geq 1-\lim_{j\rightarrow\infty}\exp{\left(-\frac{M^{2}}{C^{2}_{\alpha,p}\delta_{j}^{2\sigma(\alpha,\delta)}}\right)}=1.
$$
We see, for every $\omega\in\Sigma_{2}$, there exists a $j$ such that $\omega\in E_{j}$. Let $\delta = 2^{-(j+1)}$, and $\omega\in E_{p,q}(M,\delta)$
\end{proof}
With the fact that the union of finite zero measure sets is also a zero measure set, we finish our probabilistic estimate by combining Lemma $\ref{lemma1}$ and Lemma $\ref{lemma2}$.
\begin{Proposition}
\label{prop:prob}
Let $0<\alpha <1$, and let $2\leq p_{1},p_{2},q_{1},q_{2}\leq r<\infty$ with $\alpha p_{1}\leq2$ and $\alpha p_{2}<2$. Then we can define a set $\Sigma\subset\Omega$ satisfying $\mathcal{P}(\Sigma)=1$, and for every $\lambda>0$ and every $\omega\in\Sigma$, we can choose a $\delta(\lambda)>0$ and $M(\lambda)>0$ such that
$$\omega\in E_{p_{1},q_{1}}(M,\delta)\cap E_{p_{2},q_{2}}(\lambda,\delta) $$
\end{Proposition}
\begin{proof}
Let's fix a $\lambda>0$. we can define
$$
\Sigma_{2}=\bigcup_{j}E_{p_{2},q_{2}}(\lambda,2^{-j})
$$
Since $\alpha p_{2}<2$, by Lemma $\ref{lemma2}$, we have $\mathcal{P}(\Sigma_{2})=1$, and for every $\omega\in\Sigma_{2}$, we can find a $\delta$ depending on $\lambda$ such that $\omega\in E_{p_{2},q_{2}}(\lambda,\delta)$.
Now let's fix this $\delta$, and define
$$
\Sigma_{1}=\bigcup_{j}E_{p_{1},q_{1}}(2^{j},\delta)
$$
Since $\alpha p_{1}\leq 2$, by Lemma $\ref{lemma1}$, we have $\mathcal{P}(\Sigma_{1})=1$. Moreover, for every $\omega\in\Sigma_{1}$, we can find an $M>0$ such that $\omega\in E_{p_{1},q_{1}}(M,\delta)$. Now we define
$$
\Sigma = \Sigma_{1}\cap\Sigma_{2},
$$
which can be directly written as
\begin{equation}
\label{def:sigma}
\Sigma = \bigcup_{i=1}^{\infty}\Big[E_{p_{2},q_{2}}(\lambda,2^{-i})\bigcap\big(
\bigcup_{j=1}^{\infty}E_{p_{1},q_{1}}(2^{j},2^{-i})\big)\Big].
\end{equation}
We still have
$$
\mathcal{P}(\Sigma)\geq\mathcal{P}(\Sigma_{1})+\mathcal{P}(\Sigma_{2})-1=1
$$
Now for every $\omega\in\Sigma$ and every $\lambda>0$, we can find corresponding $\delta(\lambda,\alpha)$ and $M(\delta(\lambda,\alpha))$, such that
$$
\omega\in E_{p_{1},q_{1}}(M,\delta)\cap E_{p_{2},q_{2}}(\lambda,\delta)
$$
So the conclusion follows.
\end{proof}

\section{Energy Estimates}

Notice that the Navier-Stokes equation \eqref{eq:navier} enjoys the following natural scaling property: if $(u, p)$ is the solution of $(\ref{eq:navier})$, then the following transform
$$
u_{\lambda}(x,t)=\lambda u(\lambda x,\lambda^{2}t),\quad p_{\lambda}(x,t)=\lambda^2 p(\lambda x,\lambda^{2}t)
$$
gives another solution that satisfies this system with initial data $u_{0\lambda}(x)=\lambda u_{0}(\lambda x)$ for each $\lambda>0$.\\
\indent The spaces which are invariant under the above natural scaling are called critical spaces for the Navier-Stokes equations. For the 3D homogeneous Sobolev Spaces $\dot{H}^{s}$, the critical index $s_{c}=\frac{1}{2}$. If $s>s_{c}$ we call $\dot{H}^{s}$ subcritical, and $s<s_{c}$ we call it supercritical. Classical theory yields that if the initial data belongs to the critical
or subcritical Sobolev spaces, there exists a unique local strong solution to the Navier-Stokes equations \eqref{eq:navier}. If the initial data is in $\dot{H}^s$ for $0 \leq s < \frac{1}{2}$, Leray's result in \cite{LER} shows that one has at least one global weak solutions. If the initial data is in $H^{s}$ for $-\frac{1}{4}<s<0$, then the result in \cite{ASE} shows that almost surely the global weak solutions still exist. \\
\indent Let $v = u-g = u-e^{t\Delta}f^{\omega}$, then $v$ satisfies the following equations.
\begin{equation}
\label{eq:nonlinear}
\begin{cases}
\begin{aligned}
\partial_{t}v=\Delta v&-[\pp{P}\nabla\cdot(v\otimes v)+\pp{P}\nabla\cdot(v\otimes g)\\ &+\pp{P}\nabla\cdot(g\otimes v)+\pp{P}\nabla\cdot(g\otimes g)],
\end{aligned}\\
\nabla\cdot v=0,\\
v(x,0)=0.
\end{cases}
\end{equation}
We first derive the \textit{apriori} local-in-time energy estimate for $(\ref{eq:nonlinear})$ in the whole space case. The proof for the periodic case is similar and we omit it.
\begin{Proposition}
\label{prop:311}
Let $g=e^{t\Delta}f^{\omega}$, $0<\alpha\leq\frac{1}{2}$. Then there exists a set $\Sigma\subset\Omega$ with $\mathcal{P}(\Sigma)=1$, and for every $\omega\in\Sigma$ there exists $\lambda>0$ and a corresponding $0<\delta(\lambda)\leq 1$ such that the energy function $(\ref{energy})$ has an uniform bound,
\begin{equation}
E(v,\tau)\leq C(\alpha,\norm{f}_{\dot{H}^{-\alpha}(\bb{R}^{3})}),\quad \forall\ t\in(0,\delta],
\end{equation}
for all smooth solutions $v\in L^{\infty}([0,1],L^{2}(\bb{R}^{3}))\cap L^{2}([0,1],\dot{H}^{1}(\bb{R}^{3}))$ of $(\ref{eq:nonlinear})$.
\end{Proposition}
\begin{proof}
We make standard estimate in the energy space. For $0<t\leq\delta$, which $\delta$ to be determined later, we multiply $v$ to both side of $(\ref{eq:nonlinear})$,  by standard energy estimate in time period $[0,t]$ we get
\begin{align*}
E(v,t)=&\int_{0}^{t}\Big\{\int_{\bb{R}^{3}}2vv_{t}dx+2\norm{\nabla v}^{2}_{L^{2}(\bb{R}^{3})}\Big\}d\tau\\
=&\int_{0}^{t}\Big\{\int_{\bb{R}^{3}}2v\Delta vdx+2\norm{\nabla v}^{2}_{L^{2}(\bb{R}^{3})}-2\int_{\bb{R}^{3}}v\cdot\pp{P}\nabla\cdot(v\otimes v)dx\\
&-\Big(2\int_{\bb{R}^{3}}v\cdot\pp{P}\nabla\cdot(v\otimes g)dx+2\int_{\bb{R}^{3}}v\cdot \pp{P}\nabla\cdot(g\otimes v)dx\\
&+2\int_{\bb{R}^{3}}v\cdot \pp{P}\nabla\cdot(g\otimes g)dx\Big)\Big\}d\tau
\end{align*}
We note that
$$
\int_{\bb{R}^{3}}\left(2v\Delta v\right)dx+2\norm{\nabla v}^{2}_{L^{2}(\bb{R}^{3})}=0.
$$
Using the divergence free property of $v$, we also have
$$
\int_{\bb{R}^{3}}v\cdot\pp{P}\nabla\cdot(v\otimes v)dx=\int_{\bb{R}^{3}}v\cdot\nabla\frac{1}{2}|v|^{2}dx=0.
$$
Similarly,
$$
\int_{\bb{R}^{3}}v\cdot\pp{P}\nabla\cdot(v\otimes g)dx=\int_{\bb{R}^{3}}g\cdot\nabla\frac{1}{2}|v|^{2}dx=0.
$$
Hence, by using integration by parts, we have
\begin{equation}
\label{ensimply}
\begin{aligned}
E(v,t)=&-2\int_{0}^{t}\Big(\int_{\bb{R}^{3}}v\cdot \pp{P}\nabla(g\otimes v)dx
+\int_{\bb{R}^{3}}v\cdot \pp{P}\nabla(g\otimes g)dx\Big)d\tau\\
=&2\int_{0}^{t}\Big(\int_{\bb{R}^{3}}\nabla v: (g\otimes v)dx
+\int_{\bb{R}^{3}}\nabla v:(g\otimes g)dx\Big)d\tau\\
\leq&2\norm{\nabla v}_{L^{2}([0,t];L^{2}_{x})}\left(\norm{v\otimes g}_{L^{2}([0,t];L^{2}_{x})}+\norm{g\otimes g}_{L^{2}([0,t];L^{2}_{x})}\right).
\end{aligned}
\end{equation}

The last term in the above bracket is simply estimated as follows:
$$
\norm{g\otimes g}_{L^{2}([0,t],L^{2}(\bb{R}^{3}))}\leq \norm{g}^{2}_{L^{4}([0,t],L^{4}(\bb{R}^{3}))}.
$$
To estimate the first term in the bracket of the right hand side of $(\ref{ensimply})$, we first apply the H\"{o}lder inequality to derive that
$$
\norm{v\otimes g}_{L^{2}([0,t],L^{2}(\bb{R}^{3}))}\leq \norm{v}_{L^{6}([0,t],L^{\frac{18}{7}}(\bb{R}^{3}))}\norm{g}_{L^{3}([0,t],L^{9}(\bb{R}^{3}))}.
$$
Apply the interpolation inequality and Sobolev embedding, we have
$$
\norm{v}_{L^{6}([0,t],L^{18/7}(\bb{R}^{3}))}\leq \norm{v}_{L^{\infty}([0,t],L^{2}(\bb{R}^{3}))}^{\frac{2}{3}}\norm{v}_{L^{2}([0,t],L^{6}(\bb{R}^{3}))}^{\frac{1}{3}}
\leq \sup_{0\leq s\leq t}E^{\frac{1}{2}}(v,s).
$$
Consequently, we arrive at
\begin{equation}
\label{eq:335}
\begin{aligned}
\sup_{0\leq s\leq t}E(v,s)\leq& C\big(\sup_{0\leq s\leq t}E(v,s)\norm{g}_{L^{3}([0,t],L^{9}(\bb{R}^{3}))}\\
+&\sup_{0\leq s\leq t}E^{\frac{1}{2}}(v,s)
\norm{g}^{2}_{L^{4}([0,t],L^{4}(\bb{R}^{3}))}\big).
\end{aligned}
\end{equation}

Now our probabilistic estimate comes in. Let $(p_{1},q_{1})=(4,4)$, $(p_{2},q_{2})=(3,9)$. Since $0<\alpha\leq \frac{1}{2}$, it is clear that $\alpha p_{1}\leq 2$ and $\alpha p_{2}<2$. Let's choose a $\lambda$ small enough such that $C\lambda\norm{f}_{\dot{H}^{-\alpha}}\leq 1/2$. Now for those $(\alpha, p_{1}, q_{1}, p_{2}, q_{2}, \lambda)$, we apply Proposition $\ref{prop:prob}$ to conclude that for every $\omega\in\Sigma$ defined in Proposition $\ref{prop:prob}$, there exist $\delta(\lambda)>0$ and $M(\lambda)>0$ such that
$$\omega\in E_{3,9}(\lambda,\delta)\cap E_{4,4}(M,\delta).$$
For $t\leq\delta$, we have the following estimate
\begin{equation}
\label{eq:cases}
\begin{cases}
C\norm{g}_{L^{3}([0,t],L^{9}(\bb{R}^{3}))}\leq C\lambda\norm{f}_{\dot{H}^{-\alpha}(\bb{R}^{3})}\leq\frac{1}{2},\\
C\norm{g}_{L^{4}([0,t],L^{4}(\bb{R}^{3}))}\leq CM\norm{f}_{\dot{H}^{-\alpha}(\bb{R}^{3})}.
\end{cases}
\end{equation}
Thanks to the estimate $(\ref{eq:335})$ and the estimate $(\ref{eq:cases})$, the local energy estimate gets bounded:
$$
E(v,\tau)\leq C(\alpha,\norm{f}_{\dot{H}^{-\alpha}(\bb{R}^{3})}),(0<\tau\leq\delta).
$$
\end{proof}
\begin{Remark}
\label{remark}
Note that $\lambda$ in the above proposition depends only on $\norm{f}_{\dot{H}^{-\alpha}}$. Hence by Proposition $\ref{prop:prob}$, we see that $\delta$ in Proposition $\ref{prop:311}$ also only depends on $\alpha$ and $\norm{f}_{\dot{H}^{-\alpha}}$.
\end{Remark}

We will prove the existence of $v$ on a short time interval $[0, \delta]$ using the \textit{apriori} estimate in Proposition $\ref{prop:311}$ in next section. Now let us consider the global existence of weak solutions $u$ by assuming the existence of $v$ on $[0,\delta]$.
\begin{proof}[The proof of Theorem $\ref{thm}$]
From Proposition $\ref{prop:311}$ we have that $v\in L^{\infty}([0,\delta];L^{2})\cap L^{2}([0,\delta];\dot{H}^{1})$. So for almost every $\tau\in[\frac{\delta}{2},\delta]$, we have $v(\cdot,\tau)\in L^{2}\cap \dot{H}^{1}$. On the other hand, let's recall the deterministic estimate of $(\ref{determ})$
\begin{equation}
\label{eq:global2}
\norm{g(\cdot,\tau)}_{L^{2}}\leq C\tau^{-\frac{\alpha}{2}}\norm{f}_{\dot{H}^{-\alpha}}, \quad \norm{g(\cdot,\tau)}_{\dot{H}^{1}}\leq C\tau^{-\frac{1+\alpha}{2}}\norm{f}_{\dot{H}^{-\alpha}}.
\end{equation}
We can see $g(\tau,\cdot)\in L^{2}\cap \dot{H}^{1}$. Hence $u(\tau,\cdot)=g(\tau,\cdot)+v(\tau,\cdot)$ is also in the $L^{2}\cap \dot{H}^{1}$ for a.e. $\tau\in [\frac{\delta}{2},\delta]$. \\

Now let's take a $\delta/2<\tau_{1}<\tau_{2}<\delta$ such that $u(\cdot,\tau_{1})\in L^{2}\cap \dot{H}^{1}$, $u(\cdot,\tau_{2})\in L^{2}\cap \dot{H}^{1}$ and $u$ is a weak solution of the Navier-Stokes equations on $[\tau_{1},\tau_{2}]$. Moreover, there exists a $\epsilon_{0}>0$ such that $u$ is smooth on the interval $(\tau_{1}, \tau_{1}+\epsilon_{0})$.

\indent Let's consider the Navier-Stokes equation with initial time $\tau=\tau_{1}+\frac{1}{2}\epsilon_{0}$:
\begin{equation}
\label{eq:initialtime}
\begin{cases}
\partial_{t}u=\Delta u-\pp{P}\nabla\cdot(u\otimes u),\\
\nabla\cdot u=0,\\
u(x,\tau)=g(x,\tau)+v(x,\tau).
\end{cases}
\end{equation}
It is clear that if we can extend $u$ on $[\tau,\tau_{2}]$ to be a global weak solution on $[\tau, T]$, then $u$ will be a global weak solution on $(0,T]$ for an arbitrary $T>0$. Indeed, by the classical result of Leray $\cite{LER}$, we can have at least one weak solution in the space $L^{\infty}([\tau,T],L^{2})\cap L^{2}([\tau,T],\dot{H}^{1})$ which satisfies
$$
\norm{u(\cdot, t)}_{L^{2}_{x}}^{2}+2\int_{\tau}^{t}\norm{\nabla u(\cdot, s)}^{2}_{L^{2}_{x}}ds\leq \norm{u(\cdot, \tau)}_{L^{2}_{x}}^{2}.
$$
This gives the global existence of weak solution $u$ to the Navier-Stokes equations. Moreover, for all $t\in [\tau,T]$, we have
\begin{equation}
\label{eq:global1}
\begin{aligned}
E(u,t)\leq& \norm{u(\cdot,\tau)}_{L^{2}}\leq \norm{g(\cdot,\tau)}_{L^{2}}+\norm{v(\cdot,\tau)}_{L^{2}}\\
\leq& E(v,\tau)+C\tau^{-\frac{\alpha}{2}}\norm{f}_{\dot{H}^{-\alpha}}
=C(\alpha,\norm{f}_{\dot{H}^{-\alpha}})
\end{aligned}
\end{equation}
We remark that the bound in the above estimate only depends on $\delta$ which only depends on $\alpha$ and $\norm{f}_{\dot{H}^{-\alpha}}$ due to Remark $\ref{remark}$.
Now for every $\delta_{0}>0$ and every $t\in [\delta_{0},\tau]$,
$$E(u,t)\leq E(v,t)+E(g,t)\leq C(\delta_{0},\alpha,\norm{f}_{\dot{H}^{-\alpha}})$$

On the other hand, by Proposition $\ref{prop:311}$, it is clear that for $t\in [0,\tau]$ we have
$$E(v,t)\leq C(\alpha,\norm{f}_{\dot{H}^{-\alpha}}).$$
For $t\in [\tau,T]$, we have
$$E(v,t)\leq E(u,t)+E(g,t).$$
Hence combine $(\ref{eq:global2})$ and $(\ref{eq:global1})$. We arrive at a bound for the energy of $v$ which is independent of $\delta_{0}$:
$$E(v,t)\leq C(\alpha,\norm{f}_{\dot{H}^{-\alpha}}),\quad \forall\ t\in [0,T].$$
\end{proof}

\section{Construction of the Weak Solution}
In this section we construct the local weak solution $v$ to the perturbed Navier-Stokes equation $(\ref{eq:nonlinear})$ on using the \textit{apriori} estimate in Proposition $(\ref{prop:311})$, which is sufficient to close the proof of Theorem $\ref{thm}$ due to the discussions in section 3.

We use smooth approximation method to construct the local weak solution to the system $(\ref{eq:nonlinear})$. For any given $N>0$, we consider the following smoothed equation of $(\ref{eq:nonlinear})$ using projector $P_{\leq N}$:
\begin{equation}
\label{smooth}
\begin{cases}
\begin{aligned}
\partial_{t}v=\Delta P_{\leq N}^{2}v&-[\pp{P}\nabla\cdot P_{\leq N}(P_{\leq N}v\otimes P_{\leq N}v)+\pp{P}\nabla\cdot P_{\leq N}(P_{\leq N}v\otimes g)\\ &+\pp{P}\nabla\cdot P_{\leq N}(g\otimes P_{\leq N}v)+\pp{P}\nabla\cdot P_{\leq N}(g\otimes g)],
\end{aligned}\\
v(x,0)=0.
\end{cases}
\end{equation}
Here the smoothing projector $P_{\leq N}$ is defined in section 2.
Taking the Fourier transform to the equation $(\ref{eq:nonlinear})$, we have
\begin{equation}
\label{eq:galerkin}
\begin{cases}
\begin{aligned}
\frac{d}{dt}\widehat{v}(\xi,t)=&-|\xi|^{2}\widehat{P_{\leq N}^{2}v}(\xi,t)
-\left(I-\frac{\xi\xi^{T}}{|\xi|^{2}}\right)
\mathcal{F}\Big\{\nabla\cdot P_{\leq N}(P_{\leq N}v\otimes P_{\leq N}v)\\
&+\nabla\cdot P_{\leq N}(P_{\leq N}v\otimes g)+\nabla\cdot P_{\leq N}(g\otimes P_{\leq N}v)
+\nabla\cdot P_{\leq N}(g\otimes g)\Big\},
\end{aligned}\\
\widehat{P_{\leq N}}v(\xi,0)=0.
\end{cases}
\end{equation}
Here we dropped the dependence of $v$ on $N$ for notational convenience.

Below we are going to prove, for each fixed $N>0$, the local-in-time existence of the solution to this nonlinear ODE equation $(\ref{eq:galerkin})$ with the assumption that the force term $g$ has a good regularity.

\begin{Lemma}
Let $0<\alpha<1$, and $C_{0}>0$ be given and $f\in \dot{H}^{-\alpha}$. Assume that $g$ satisfies
\begin{equation}
\label{section4}
\begin{cases}
\norm{g(\cdot,t)}_{L^{2}_{x}}\leq C_{0}t^{-\frac{\alpha}{2}}\norm{f}_{\dot{H}^{-\alpha}},\\
\norm{g}_{L^{4}([0,t],L^{4}_{x})}\leq C_{0}\norm{f}_{\dot{H}^{-\alpha}}.
\end{cases}
\end{equation}
We define the space
$$X_{t}=C([0,t],L^{2}_{\xi})\cap L^{2}([0,t],|\xi|L^{2}_{\xi})$$
with the norm
$$
\norm{h}_{X_{t}}= \sup_{t}\norm{h}_{L^{2}_{\xi}}+\norm{\xi h}_{L^{2}_{t}L^{2}_{\xi}}.
$$
Then there exists a $\delta>0$ such that the truncated ODE system $(\ref{eq:galerkin})$ has a unique solution in $X_{\delta}$.
\end{Lemma}
\begin{proof}
In this proof we apply the fixed point argument. Let's define the map $I$ via the equation $(\ref{eq:galerkin})$:
\begin{align*}
I(\widehat{v}(\xi,t))=&\int_{0}^{t}-|\xi|^{2}\widehat{P^{2}_{\leq N}v}(\xi,s)+\left(I-\frac{\xi\xi^{T}}{|\xi|^{2}}\right)
\rho(\xi/N)\big(\xi\cdot\mathcal{F}(P_{\leq N}v\otimes P_{\leq N}v)\\
&+\xi\cdot\mathcal{F}(P_{\leq N}v\otimes g)+\xi\cdot\mathcal{F}(g\otimes P_{\leq N}v)+\xi\cdot\mathcal{F}(g\otimes g)\big)ds
\end{align*}

With a direct calculation we can have
$$
\norm{\int_{0}^{t}-|\xi|^{2}\widehat{P_{\leq N}^{2}v}(\xi,s)ds}_{L^{2}_{\xi}}\leq N^{2}t\norm{\widehat{v}(\xi,s)}_{L^{2}_{\xi}}.
$$
Using the Young inequality, we get
\begin{align*}
&\norm{\int_{0}^{t}\rho(\xi/N)\xi\cdot\mathcal{F}(P_{\leq N}v\otimes P_{\leq N}v)ds}_{L^{2}_{\xi}}\\
\leq& N\int_{0}^{t}\norm{\widehat{P_{\leq N}v}}_{L^{1}_{\xi}}\norm{\widehat{P_{\leq N}v}}_{L^{2}_{\xi}}ds\\
\leq& N^{5/2}t\norm{\widehat{v}}_{L^{\infty}([0,t],L^{2}_{\xi})}^{2}.
\end{align*}

Notice the assumption $(\ref{section4})$, $\norm{g(\cdot,t)}_{L^{2}}\leq Ct^{-\frac{\alpha}{2}}\norm{f}_{\dot{H}^{-\alpha}}$. Along with the Minkowski inequality and Plancherel equality we can estimate that
\begin{align*}
&\norm{\int_{0}^{t}\rho(|\xi|/N)\xi\cdot\mathcal{F}(P_{\leq N}v\otimes g+g\otimes P_{\leq N}v)ds}_{L^{2}_{\xi}}\\
\leq& 2N\int_{0}^{t}\norm{\widehat{g}}_{L^{2}_{\xi}}\norm{\widehat{P_{\leq N}v}}_{L^{1}_{\xi}}ds\\
\leq &2N^{\frac{5}{2}}\norm{\widehat{v}}_{L^{\infty}([0,t],L^{2}_{\xi})}\norm{g}_{L^{1}_{t}L^{2}_{x}}\\
\leq& 2N^{\frac{5}{2}}\norm{\widehat{v}}_{L^{\infty}([0,t],L^{2}_{\xi})}t^{1-\frac{\alpha}{2}}\norm{f}_{\dot{H}^{-\alpha}}.
\end{align*}
Together with the assumption $\norm{g}_{L^{4}([0,t],L^{4}_{x})}\leq C\norm{f}_{\dot{H}^{-\alpha}}$ in $(\ref{section4})$, also with the help of Plancherel Equality we have
\begin{align*}
&\norm{\int_{0}^{t}\rho(|\xi|/N)\xi\cdot\widehat{g\otimes g}ds}_{L^{2}_{\xi}}\leq N\int_{0}^{t}\norm{\widehat{g\otimes g}}_{L^{2}_{\xi}}ds\\
&\leq N\int_{0}^{t}\norm{g\otimes g}_{L^{2}_{x}}ds
\leq N\int_{0}^{t}\norm{g}^{2}_{L^{4}_{x}}ds\\
&\leq Nt^{\frac{1}{2}}\norm{g}^{2}_{L^{4}L^{4}}
\leq NC_{0}^{2}\norm{f}^{2}_{\dot{H}^{-\alpha}}t^{\frac{1}{2}}.
\end{align*}

Now we have the $L^{\infty}_{t}L^{2}_{\xi}$ norm estimate of $I(\widehat{v})$.
\begin{align*}
\norm{I(\widehat{v}(\xi,s))}_{L^{\infty}([0,t],L^{2}_{\xi})}\leq& C\{N^{2}t\norm{\widehat{v}}_{L^{\infty}([0,t],L^{2}_{\xi})}
+N^{\frac{5}{2}}t\norm{\widehat{v}}_{L^{\infty}([0,t],L^{2}_{\xi})}^{2}\\
&+N^{\frac{5}{2}}t^{1-\frac{\alpha}{2}}C_{0}\norm{f}_{\dot{H}^{-\alpha}}\norm{\widehat{v}}_{L^{\infty}([0,t],L^{2}_{\xi})}+
NC_{0}^{2}t^{\frac{1}{2}}\norm{f}^{2}_{\dot{H}^{-\alpha}}\}
\end{align*}
By similar argument, we can have the $L^{2}([0,t],\xi L^{2}_{\xi})$ estimate
\begin{align*}
\norm{\xi I(\widehat{v}(\xi,s))}_{L^{2}([0,t],L^{2}_{\xi})}\leq& C\{N^{3}t^{\frac{3}{2}}\norm{\widehat{v}}_{L^{\infty}([0,t],L^{2}_{\xi})}
+N^{\frac{7}{2}}t^{\frac{3}{2}}\norm{\widehat{v}}_{L^{\infty}([0,t],L^{2}_{\xi})}^{2}\\
&+N^{\frac{7}{2}}t^{\frac{3-\alpha}{2}}C_{0}\norm{f}_{\dot{H}^{-\alpha}}\norm{\widehat{v}}_{L^{\infty}([0,t],L^{2}_{\xi})}+
N^{2}C_{0}^{2}t\norm{f}^{2}_{\dot{H}^{-\alpha}}\}
\end{align*}
By the result of the last two estimates we can see that $I$ is a continuous map from a ball $\{f\in X_{t}:\norm{f}_{X_{t}}\leq M_{0}\norm{f}^{2}_{\dot{H}^{-\alpha}}\}$ into itself if $t$ is small enough. Moreover, it is easy to see that we can choose an even smaller $t_{1}=t_{1}(N)$ such that $I$ is a contraction in this ball. Hence by the fixed point theorem there exists an unique point which we still denote as $\widehat{v}$ such that $I(\widehat{v})=\widehat{v}$, which is of course the solution to the equation $(\ref{eq:galerkin})$ on the time period $[0,t_{1}]$ for fixed $N>0$.
\end{proof}

Now we can show the existence of the local weak solution to the perturbed Navier-Stokes equation $(\ref{eq:nonlinear})$. We let $g=e^{t\Delta}f^{\omega}$, and let $\Sigma$ be the set defined in Proposition $\ref{prop:prob}$. Then for any $\omega\in\Sigma$, we have  $\norm{g}_{L^{4}_{t}L^{4}_{x}}\leq C\norm{f}_{\dot{H}^{-\alpha}}$. On the other hand, by the deterministic estimate $(\ref{determ})$ of $g$ we can also have $\norm{g(\cdot,t)}_{L^{2}}\leq Ct^{-\frac{\alpha}{2}}\norm{f}_{\dot{H}^{-\alpha}}$. Hence $g$ satisfies the assumption $(\ref{section4})$. Thus by Lemma 4.1, for any fixed $N>0$, there exists a $t_{1}>0$ and a solution $\widehat{v}$ for the system $(\ref{eq:galerkin})$ in $X_{t_{1}}$. This provides a solution $v^{(N)}$ to the truncated system $(\ref{smooth})$ which depends on $N$ and is an approximated solution of equation $(\ref{eq:nonlinear})$.\\

We are going to show the existence of a local weak solution to the equation $(\ref{eq:nonlinear})$ by the approximated solution sequence $\{v^{(N)}\}$. Define the energy space
$$
\pp{E}(t)=L^{\infty}([0,t],L^{2})\cap L^{2}([0,t],\dot{H}^{1}),
$$
with the norm $\norm{v}_{\pp{E}(t)}=E(v,t).$
By a similar argument as in deriving the \textit{apriori} energy estimate in section 3,
$v^{(N)}$ can be extended to $[0,T]$ for all $T>0$, and we have
\begin{equation}
\label{eq:vn}
E(v^{(N)}, t)\leq C(\alpha,\norm{f}_{\dot{H}^{-\alpha}}),\quad \forall\ t\in(0,T].
\end{equation}

Now let's fix the time $T$, and estimate $\norm{\partial_{t}v^{(N)}}_{L^2_tH^{-2}_{\rm loc}}$. Let $B$ be an arbitrary bounded domain with smooth boundary in $\mathbb{R}^{3}$. By the definition of negative-order Sobolev space, one has
$$
\norm{\partial_{t}v^{(N)}}_{L^2([0,T];H^{-2}(B))}=\sup_{\substack{
\phi\in L^{2}_{t}H^{2}_{0}(B)\\
\nabla\cdot\phi=0
}}\int_{0}^{T}\int_{B}\partial_{t}v^{(N)}\phi dxdt.
$$
It is easy to see that
$$
\int_{0}^{T}\int_{B}\Delta P_{\leq N}^{2}v\phi dxdt\leq \norm{\nabla P_{\leq N}v}_{L^{2}L^{2}}\norm{\phi}_{L^{2}\dot{H}^{1}}\leq \sup_{t\in [0,T]}E(P_{\leq N}v,t)\norm{\phi}_{L^{2}\dot{H}^{1}}.
$$
We note that the Leray projector $\pp{P}$ is symmetric, and $\phi$ is a divergence-free function. Hence for any function $h$ we have
$$
\int_{\bb{R}^{3}} <\pp{P}h,\phi>dx = \int_{\bb{R}^{3}} <h,\pp{P}\phi>dx=\int_{\bb{R}^{3}} <h,\phi>dx
$$
Thanks to this property we can derive that
\begin{align*}
&\Big|\int_{0}^{T}\int_{B}\pp{P}[\nabla\cdot P_{\leq N}(P_{\leq N}v\otimes P_{\leq N}v)]\cdot\phi dxdt\Big|\\
=&\Big|\int_{0}^{T}\int_{\bb{R}^{3}}\pp{P}[\nabla\cdot P_{\leq N}(P_{\leq N}v\otimes P_{\leq N}v)]\cdot\phi dxdt\Big|\\
=&\Big|\int_{0}^{T}\int_{\bb{R}^{3}} P_{\leq N}(P_{\leq N}v\otimes P_{\leq N}v):\nabla\phi dxdt\Big|\\
\leq & \int_{0}^{T}\norm{P_{\leq N}v}_{L^{2}_{B}}\norm{P_{\leq N}v}_{L^{3}_{B}}\norm{\nabla\phi}_{L^{6}_{B}}dt\\
\leq& T^{\frac{1}{4}}\norm{P_{\leq N}v}^{\frac{3}{2}}_{L^{\infty}_{t}L^{2}_{B}}\big(\int_{0}^{T}\norm{P_{\leq N}v}^{2}_{L^{6}_{B}}dt\big)^{\frac{1}{4}}
\big(\int_{0}^{T}\norm{\nabla\phi}_{L^{6}_{B}}^{2}dt\big)^{\frac{1}{2}}\\
\leq& C_{B}T^{\frac{1}{4}}\sup_{t\in [0,T]}E(P_{\leq N}v,t)^{2}\norm{\phi}_{L^{2}_{t}\dot{H}^{2}_{B}}.
\end{align*}
With the similar argument, we can get
\begin{align*}
&\Big|\int_{B}\pp{P}[\nabla\cdot P_{\leq N}(P_{\leq N}v\otimes g)]\cdot\phi dx\Big|\leq C_{B}T^{\frac{1}{4}}\sup_{t\in [0,T]}E(P_{\leq N}v,t) \norm{g}_{L^{4}_{t}L^{3}_{B}}\norm{\phi}_{L^{2}_{t}\dot{H}^{2}_{B}},\\
&\Big|\int_{B}\pp{P}[\nabla\cdot P_{\leq N}(g\otimes P_{\leq N}v)]\cdot\phi dx\Big|\leq C_{B}T^{\frac{1}{4}}\sup_{t\in [0,T]}E(P_{\leq N}v,t) \norm{g}_{L^{4}_{t}L^{3}_{B}}\norm{\phi}_{L^{2}_{t}\dot{H}^{2}_{B}},\\
&\Big|\int_{B}\pp{P}[\nabla\cdot P_{\leq N}(g\otimes g)]\cdot\phi dx\Big|\leq C_{B}\norm{g}^{2}_{L^{4}_{t}L^{4}_{B}}\norm{\phi}_{L^{2}_{t}\dot{H}^{1}_{B}}.
\end{align*}

Since $v^{(N)}$ is a solution of the equation $\ref{smooth}$, for any bounded domain $B$ with smooth boundary, we have
\begin{eqnarray}\nonumber
\|\partial_tv^{(N)}\|_{L^2_tH^{-2}(B)} &\leq& \|\Delta P_{\leq N}^{2}v\|_{L^2_tH^{-2}(B)} + \|\nabla\cdot P_{\leq N}(P_{\leq N}v\otimes P_{\leq N}v)\|_{L^2_tH^{-2}(B)}\\\nonumber &&+\ \|\nabla\cdot P_{\leq N}(P_{\leq N}v\otimes g)\|_{L^2_tH^{-2}(B)} + \|\nabla\cdot P_{\leq N}(g\otimes P_{\leq N}v)\|_{L^2_tH^{-2}(B)}\\\nonumber
&&+\ \|\nabla\cdot P_{\leq N}(g\otimes g)]\|_{L^2_tH^{-2}(B)}\\\nonumber
&\leq& C_{B}\big(\sup_{t\in [0,T]}E(P_{\leq N}v,t)+T^{\frac{1}{4}}\sup_{t\in [0,T]}E(P_{\leq N}v,t)^{2}\\\nonumber
&&+\ T^{\frac{1}{4}}\sup_{t\in [0,T]}E(P_{\leq N}v,t) \norm{g}_{L^{4}_{t}L^{3}(B)}+\norm{g}^{2}_{L^4_tL^{4}(B)}\big)\\\nonumber
&&\ \big(\norm{\phi}_{L^2_t\dot{H}^{1}(B)}+\norm{\phi}_{L^2_t\dot{H}^{2}(B)}\big).
\end{eqnarray}
Note that $\norm{g}_{L^{4}_{t}L^{4}_{x}}\leq C\norm{f}_{\dot{H}^{-\alpha}}$, and $\norm{g}_{L^{4}_{t}L^{3}(B)}\leq C_{B}\norm{g}_{L^{4}_{t}L^{4}(B)}$. By the energy estimate $(\ref{eq:vn})$, we derive that
\begin{equation}
\|\partial_tv^{(N)}\|_{L^2_tH^{-2}(B)} \leq C_{B}(\alpha, \norm{f}_{\dot{H}^{\alpha}},T^{\frac{1}{4}})\ ,\quad \forall\ t\in(0,T].
\end{equation}
This estimate implies that
$$
\partial_{t}v^{(N)}\in L^{2}([0,T];H^{-2}_{\rm loc}).
$$
Hence, by the standard compactness argument, see Temam $\cite{TER}$. one has
\begin{equation}\label{convergence}
v^{(N)} \rightarrow v\ \ {\rm in}\ C([0, t], H^{-2}_{\rm loc}),\quad v^{(N)} \rightharpoonup v\ \ {\rm in}\ \pp{E}(t)\ ,\quad \forall\ t\in(0,T].
\end{equation}
Moreover, since $\norm{v^{(N)}}_{\pp{E}(t)}$ is uniform bounded, we have
$$
\norm{v}_{\pp{E}(t)}\leq C(\alpha,\norm{f}_{\dot{H}^{-\alpha}}),\quad \forall\ t\in(0,T].
$$\\

Next we prove an interpolation inequality:
\begin{equation}
\label{ineq:interpolation}
\|v^{(N)}-v\|_{L^{2}([0, \delta], L^{2}_{\rm loc})}\leq C\|v^{(N)}-v\|_{L^{2}([0, \delta], H^{-2}_{\rm loc})}^{\frac{1}{3}}\|v^{(N)}-v\|_{L^{2}([0, \delta], \dot{H}^{1})}^{\frac{2}{3}}.
\end{equation}
To show this inequality, take $B$ be an arbitrary bounded domain with smooth boundary, and define
$$
\phi = \begin{cases}
1\quad (x\in B),\\
0\quad (x\in (2B)^{c}).
\end{cases}
$$
Then by the interpolation theorem for the whole space, we have
\begin{align*}
\|v^{(N)}-v\|_{L^{2}(B)}&=\|(v^{(N)}-v)\phi\|_{L^{2}(\bb{R}^{3})}\leq \|(v^{(N)}-v)\phi\|_{\dot{H}^{-2}(\bb{R}^{3})}^{\frac{1}{3}}\|(v^{(N)}-v)\phi\|_{\dot{H}^{1}(\bb{R}^{3})}^{\frac{2}{3}}\\
&\leq \|(v^{(N)}-v)\phi\|_{\dot{H}^{-2}(2B)}^{\frac{1}{3}}\Big(\|v^{(N)}-v\|_{\dot{H}^{1}(2B)} +\|v^{(N)}-v\|_{L^{2}(2B)}\Big)^{\frac{2}{3}}\\
&\leq \|(v^{(N)}-v)\phi\|_{\dot{H}^{-2}(2B)}^{\frac{1}{3}}\Big(\|v^{(N)}-v\|_{\dot{H}^{1}(2B)} +C_{B}\|v^{(N)}-v\|_{L^{6}(2B)}\Big)^{\frac{2}{3}}\\
&\leq \|(v^{(N)}-v)\phi\|_{\dot{H}^{-2}(2B)}^{\frac{1}{3}}\|v^{(N)}-v\|_{\dot{H}^{1}(\bb{R}^{3})}^{\frac{2}{3}}
\end{align*}
Note that for any $\psi\in H^{2}_{0}(2B)$, we have
\begin{align*}
\int_{2B}(v^{(N)}-v)\phi\psi dx \leq \|v^{(N)}-v\|_{\dot{H}^{-2}(2B)}\norm{\phi\psi}_{\dot{H}^{2}(2B)}\leq C\|v^{(N)}-v\|_{\dot{H}^{-2}(2B)}.
\end{align*}
Integral with $t$,
\begin{align*}
&\int_{0}^{T}\|v^{(N)}-v\|^{2}_{L^{2}(B)}dt\leq C\int_{0}^{T}\|v^{(N)}-v\|_{\dot{H}^{-2}(2B)}^{\frac{2}{3}}\|v^{(N)}-v\|_{\dot{H}^{1}(\bb{R}^{3})}^{\frac{4}{3}}dt\\
&\leq C\|v^{(N)}-v\|_{L^{2}_{t}\dot{H}^{-2}(2B)}^{\frac{2}{3}}\|v^{(N)}-v\|_{L^{2}_{t}\dot{H}^{1}(\bb{R}^{3})}^{\frac{4}{3}}.
\end{align*}
Thus the inequality $(\ref{ineq:interpolation})$ follows.\\

Hence, using the $(\ref{convergence})$, we have
\begin{equation}\label{convergence-1}
\begin{aligned}
&\|v^{(N)} - v\|_{L^3_t(L^3_{\rm loc})} \leq C\|v^{(N)} - v\|_{L^2_t(L^2_{\rm loc})}^{\frac{1}{6}}\|v^{(N)} - v\|_{L^{\frac{10}{3}}_{t}(L^{\frac{10}{3}}_{x})}^{\frac{5}{6}}\\
&\leq C\|v^{(N)} - v\|_{L^2_t(L^2_{\rm loc})}^{\frac{1}{6}}\|v^{(N)} - v\|_{\pp{E}(t)}^{\frac{5}{6}}\\
&\longrightarrow 0\ {\rm as}\ N\to \infty.
\end{aligned}
\end{equation}
As a result, using $(\ref{convergence})$ and $(\ref{convergence-1})$, for any $\phi \in C_0^\infty(\mathbb{R}_+\times\mathbb{R}^3)$ with $\nabla\cdot \phi = 0$, one may take the limit $N \to \infty$ in the following identity
\begin{align*}
&\iint\partial_{t}\phi v^{(N)}+\iint\Delta\phi P_{\leq N}^{2}v^{(N)}\\
&=-\iint[P_{\leq N}(P_{\leq N}v^{(N)}\otimes P_{\leq N}v^{(N)}):\nabla\phi+ P_{\leq N}(P_{\leq N}v^{(N)}\otimes g):\nabla\phi\\ &+P_{\leq N}(g\otimes P_{\leq N}v^{(N)}):\nabla\phi+P_{\leq N}(g\otimes g):\nabla\phi],
\end{align*}
to derive that
\begin{align*}
&\iint\partial_{t}\phi v+\iint\Delta\phi v\\
&=-\iint[v\otimes v:\nabla\phi+ (v\otimes g):\nabla\phi\\ &+(g\otimes v):\nabla\phi+(g\otimes g):\nabla\phi].
\end{align*}
This shows that $v$ is a weak solution to $(\ref{eq:nonlinear})$ on $[0,T]$. In particular, $T$ can be taken as $\delta$ and thus we get the local weak solution.

\section*{Acknowledgement}
The authors greatly thank Prof. Zhen Lei for many constructive discussions. The authors were in part supported by NSFC (grants No. 11171072, 11421061, 11222107 and 11301338), Shanghai Shu Guang project, Shanghai Talent Development Fund and SGST 09DZ2272900.

\end{document}